\documentclass[12pt]{amsproc}

\usepackage{latexsym,amssymb}
\usepackage[english]{babel}
\def\PSL{{\rm PSL}}
\def\SL{{\rm SL}}
\newcommand{\tr}{\mathop{\rm tr}\nolimits}
\newcommand{\dete}{\mathop{\rm det}\nolimits}
\newcommand{\diago}{\mathop{\rm diag}\nolimits}

\newtheorem{theorem}{Theorem}[section]
\newtheorem{lemma}[theorem]{Lemma}
\newtheorem{proposition}[theorem]{Proposition}
\newtheorem{corollary}[theorem]{Corollary}
\newtheorem{conjecture}[theorem]{Conjecture}

\begin{document}

\title{Coprime commutators in $\PSL(2,q)$} 

\author{Marco Antonio Pellegrini}
\address{Departamento de Matem\'atica\\ 
Universidade de Bras\'ilia\\
  Bras\'ilia - DF, Brazil }
\email{pellegrini@unb.br}

\author{Pavel Shumyatsky}
\address{Departamento de Matem\'atica\\ 
Universidade de Bras\'ilia\\
  Bras\'ilia - DF, Brazil }
\email{pavel@mat.unb.br}

\thanks{The second author was supported by CNPq-Brazil}


\begin{abstract}
\noindent We show that every element of $\PSL(2,q)$ is a commutator of elements of coprime orders. This is proved by showing first that in $\PSL(2,q)$ any two involutions are conjugate by an element of odd order. 
\end{abstract}

\maketitle

\section{Introduction}

An element $g$ of a group $G$ is called commutator if there exist $x,y\in G$ such that $g=[x,y]$. Here, as usual, $[x,y]=x^{-1}y^{-1}xy$. 
In 1951 Ore conjectured that every element of a nonabelian finite simple group is a commutator. 
Almost sixty years later, as a result of major efforts by many group-theorists, Ore's conjecture has been confirmed by Liebeck, O'Brien, Shalev and Tiep \cite{lost}. 

An element of a group is called a \emph{coprime} commutator if it can be written as a commutator of elements of coprime orders. In \cite{eu} the second author of the present paper conjectured that every element of a nonabelian finite simple group is a coprime commutator.  He showed that this is true for the alternating groups.
Computational work with \textsc{Magma} \cite{magma} seems to confirm this conjecture. 
Namely, we verified that every element of a nonabelian simple group of order less than $10^7$
is a coprime commutator. Here we will prove that this is also true for all simple groups $\PSL(2,q)$.

\begin{theorem}\label{commu} 
Let $q>3$ be a prime-power. Every element of $\PSL(2,q)$ is a coprime commutator.
\end{theorem}

Our proof of the above theorem is based on analysis of cosets of a certain subgroup in $\SL(2,q)$, where $q\equiv 1 \pmod 4$. 
We will show that every coset of that subgroup contains an element of odd order. From this we deduce the following theorem.

\begin{theorem}\label{cose} 
Each coset of the centralizer of an involution in $\PSL(2, q)$ contains an element of odd order.
\end{theorem}

The above theorem is somewhat related to the following question asked by Paige in the beginning of the sixties:

\medskip
\noindent {\it Is it true that if $T$ is a Sylow $2$-subgroup of the finite group $G$, then each coset of $T$ in $G$ contains at least one element of odd order?}
\medskip

Thompson gave a negative answer to Paige's question in \cite{thompson}. He showed that the group $\PSL(2,53)$ provides a counter-example. 
Recently Goldstein and Guralnick proved that for any prime $p$ there exist infinitely many finite simple groups $G$ with a coset of a Sylow $p$-subgroup $T$ of $G$ in which every element has order divisible by $p$ \cite{goldgura}. 

In view of our Theorem \ref{cose} the following related conjecture seems plausible.

\begin{conjecture}\label{conj} 
Let $T$ be a Sylow $2$-subgroup of a finite group $G$ and $t$ an involution in $Z(T)$.
Then each coset of $C_G(t)$ contains an element of odd order.
\end{conjecture}

It is not difficult to see that in the case of soluble groups the conjecture is true. Our Theorem \ref{cose} shows that the conjecture is also true when $G=\PSL(2, q)$. Note that if we allow $t$ to be non-central in $T$, then there are counter-examples.
For instance take $t=(1,2)$ in $G=Sym(4)$. 
In this case, the coset $C_G(t)(1,3,2,4)$ consists only of elements of even order.

However it seems that for finite simple groups Conjecture \ref{conj} can be generalized in the following way.

\begin{conjecture}\label{conj2} 
Each coset of the centralizer of an involution in a finite simple group $G$ contains an element of odd order unless 
 $G=\PSL(n,2)$ with $n\geq 4$.
\end{conjecture}

We verified with \textsc{Magma} that the last conjecture holds for all simple groups of order less than $10^{10}$. Furthermore, we will show that the groups $\PSL(n,2)$, with $n\geq 4$, always are an exception to Conjecture \ref{conj2}.

\section{Cosets in $\PSL(2,q)$}

In this section we prove Theorem \ref{cose}. Recall that in $\PSL(2,q)$ all involutions are conjugate (see, for instance, \cite[\S 38]{dorn}). Thus, it suffices to prove the claim for a single involution.

First however some preparatory work is required. We start with two elementary lemmas.

\begin{lemma}\label{15}
Let $a,b,x,y$ be non-zero elements of a field and suppose that $xa+ya^{-1}=xb+yb^{-1}$. Then, either $a=b$ or $xa=yb^{-1}$.
\end{lemma}

\begin{proof}
Multiplying both sides of the equation  $xa+ya^{-1}=xb+yb^{-1}$ by $xy^{-1}ab^{-1}$,
we have
$$x^2y^{-1}a^2b^{-1}+xb^{-1}=x^2y^{-1}a+xab^{-2}.$$
Thus, we deduce
$$x^2y^{-1}a^2b^{-1}-x^2y^{-1}a =xab^{-2}-xb^{-1}.$$
Therefore, 
  $$x^2y^{-1}a(ab^{-1}-1)=xb^{-1}(ab^{-1}-1)$$
and so the lemma follows.
\end{proof}

\begin{lemma}\label{16}
Let $r$ be a prime-power number and $F$ a finite field with $r$ elements.
For any  non-zero element $u$ of  $F$, set
$$S_u=\{ a+b \mid a,b\in F, ab=u\}.$$
Then
\begin{enumerate}
\item $|S_u|=\frac{r-2}{2}+1$, if $r$ is even;
\item $|S_u|=\frac{r-3}{2}+2$, if $r$ is odd and $u$ is a square in $F$;
\item $|S_u|=\frac{r-1}{2}$, if $r$ is odd and $u$ is not a square in $F$.
\end{enumerate}
\end{lemma}

\begin{proof}
Suppose $s\in S_u$. Then $s$ can be written in the form $s_a=a+ua^{-1}$.
If $u$ is a square in $F$, choose $d$ such that $u=d^2$. 
If $r$ is even, then for every $u$ there is a unique $d$ such that $u=d^2$. 
If $r$ is odd, then either $u$ is non-square or there are precisely two elements, $d$ and $-d$, with the above property. 
According to Lemma \ref{15} for every possible value of $s$, different from $s_d$ and $s_{-d}$, there are precisely two elements $a_1,a_2\in F^\times$ such that $s=s_{a_1}$ and $s=s_{a_2}$.  Now we let $a$ run over $F^\times\setminus\{d,-d\}$. 

If $r$ is even, we obtain $\frac{r-2}{2}$ different values for $s_a$. Adding to this set $s_d$, we conclude that $|S_u|=\frac{r-2}{2}+1$.

If $r$ is odd and $u$ is a square, we obtain $\frac{r-3}{2}$ different values for $s_a$. 
Adding to this set $s_d$ and $s_{-d}$, we conclude that $|S_u|=\frac{r-3}{2}+2$.

Finally, if $r$ is odd and $u$ is not a square, we obtain $\frac{r-1}{2}$ different values for $s_a$. Therefore in this case $|S_u|=\frac{r-1}{2}$.
\end{proof}

We can now prove the following.

\begin{proposition}\label{propH} Let $K$ be the finite field with $q$ elements, where $q\equiv 1 \pmod 4$. Let $ \widetilde G=\SL(2,K)$ and choose a generator $\nu$ of the multiplicative group $K^\times$ of $K$. Let
$$a=\left(\begin{array}{cc}\nu & 0 \\ 0 & \nu^{-1}\end{array}\right), \qquad b=\left(\begin{array}{cc} 0 & 1 \\ -1 & 0 \end{array}\right).$$
Denote by $H$ the subgroup of $\widetilde G$ generated by $a$ and $b$. Then for every $x\in \widetilde G\setminus H$ the coset $Hx$ contains an element of odd order.
\end{proposition}

\begin{proof}
Suppose that the proposition is false and the coset $Hx$ entirely consists of elements of even order. 
Then in fact every element in $Hx$ has order divisible by $4$. 
Indeed, suppose that the order of $x$ is not divisible by $4$. Write $\langle x\rangle =\langle y\rangle\times\langle z\rangle$, where $y$ has odd order and $z$ is an involution such that $x=yz$. 
Then $z= \bigl(\begin{smallmatrix} -1 & 0 \\0 & -1 \end{smallmatrix} \bigr)\in H$. Therefore $Hx$ contains the element $y$ which is of odd order. 
Hence, we assume that all elements in $Hx$ have order divisible by $4$.

We will now use the fact that every element of $\widetilde G$ whose order is divisible by $4$ is conjugate to an element of $H$ (see \cite[\S 38]{dorn}). Let $S_u$ have the same meaning as in Lemma \ref{16} and $\tr(h)$ denote the trace of a matrix $h$. Then $S_1$ is precisely the set $\{\tr(h) \mid h\in H\}$. Here we use the fact that $0\in S_1$, since $q\equiv 1 \pmod 4$. 
Further, Lemma \ref{15} shows that the order of $h\in \langle a\rangle$ is completely determined by $\tr(h)$. 
Let 
$$S^*=\{\tr(h) \mid \text{ $h=1$ or $h$ is of even order in } H\}.$$ 
Thus, we will obtain a contradiction once we show that there exists $h\in H$ such that $\tr(hx)\not\in S^*$.
Let $x=\bigl(\begin{smallmatrix} \alpha & \beta \\\gamma & \delta \end{smallmatrix} \bigr)$ and denote by $R$ the set $\{\tr(a^i x)\mid i=0,1,\dots,q-2\}$.
Then 
$$R=\{\alpha\nu^i+\delta\nu^{-i} \mid i=0,1,\dots,q-2\}=S_{\alpha\delta}.$$ 
Suppose that $\alpha\delta$ is not a square in $K$. 
Then, by Lemma \ref{16} $|S_{\alpha\delta}|=\frac{q-1}{2}$, while $|S_1|=\frac{q-3}{2}+2$. 
Hence $\tr(a^i x) \not\in S^*$ for some $i$, as required. 
Therefore we assume that $\alpha\delta =m^2$ for some $m\in K$. 

If $m=0$, then $\beta\gamma \neq 0$ and we can work with the matrix 
$bx=\bigl(\begin{smallmatrix} \gamma & \delta \\-\alpha & -\beta \end{smallmatrix} \bigr)$ 
in place of $x$. So without loss of generality we can assume that $m\neq 0$. We need to show that $R\neq S^*$. If $q-1$ is not a $2$-power, the subgroup $\langle a\rangle$ contains elements of odd order and therefore $S^*\neq S_1$. Since by Lemma \ref{16} $|S_1|=|R|$, we obtain a contradiction. 
Thus, we assume that $q-1$ is a $2$-power and $S^*=S_1=R$. We have $$S_1=R=S_{m^2}=\{m\nu^i+ m\nu^{-i}\mid i=0,\dots,q-2\}=mS_1.$$ 

Suppose that $m^2=1$. Since $\dete(x)=1$, it follows that either $\beta=0$ or $\gamma=0$.
If $\beta=\gamma=0$, then $x\in\langle a\rangle$. 
If $\beta\neq 0$, we see that the coset $\langle a\rangle x$ contains the transvection $\bigl(\begin{smallmatrix} 1 & \beta \\0 & 1 \end{smallmatrix} \bigr)$, which is of odd order. 
Thus, without loss of generality we can assume that $m^2\neq 1$. Since $q\equiv 1 \pmod 4$, the field $K$ contains an element $j$ such that $j^2=-1$.
We know that $S_1=mS_1$ and since the order of $m$ is at least $4$, 
it follows that $S_1=jS_1$. Recall that $S_1=\{ k+k^{-1}\mid k\in K^\times\}$ and it is easy to see that 
$$jS_1=\{j k+ jk^{-1}\mid k\in K^\times\}=\{k-k^{-1}\mid k\in K^\times\}.$$ 
We therefore deduce that 
$$\{k+k^{-1}\mid k\in K^\times\}= \{k-k^{-1}\mid k\in K^\times\}.$$ 
Considering now the set of squares of the above set we conclude that 
$$\{k^2+k^{-2}+2\mid k\in K^\times\}=\{k^2+k^{-2}-2 \mid k\in K^\times\}. $$ 
Let $T=\{k^2+k^{-2}\mid k\in K^\times\}$. The above equality shows that $T+4=T$. 
If $q=5$, then $T=\{0,\nu,\nu^3\}$ and the equality $T+4=T$ yields a contradiction.
Thus $q\neq 
5$ and so $j$ is a square in $K$.  Let us determine $|T|$. 

It is clear that the order of $T$ is the same as the order of $\{k^2+k^{-2}+2 \mid k\in K^\times\}$. The set $\{k^2+k^{-2}+2\mid  k\in K^\times\}$ is precisely the set of all squares of elements of $S_1$. 
Obviously $S_1=-S_1$ and so the order of the set of all squares of non-zero elements of $S_1$ is half of $(|S_1|-1)$. 
Since $T$ also contains $0$, Lemma \ref{16} shows that $|T|=\frac{q+3}{4}$. Let $p$ be the characteristic of the field $K$. The equality $T+4=T$ shows that $|T|$ must be divisible by $p$ and since $|T|=\frac{q+3}{4}$, we conclude that $p=3$. By Mih\v ailescu's theorem on Catalan's conjecture \cite{catalan}, it now follows that $q=9$.
However, a direct computation shows that in $\SL(2,9)$ every coset of $H$ contains an element of odd order. More precisely, the computation shows that the set of the orders of elements in $Hx$ is necessarily one of the following:
$$
\{ 5, 8, 10 \},\quad \{ 3, 4, 6, 8  \},\quad
\{ 3, 4, 5, 6, 10 \},\quad \{ 3, 4, 5, 6, 8, 10 \}.$$
This completes the proof.
\end{proof}

Now we are ready to prove Theorem \ref{cose}.

\begin{proof}[Proof of Theorem \ref{cose}] 
Let $C$ be the centralizer in $G=\PSL(2,q)$ of an involution.
If $q$ is even, then the elements of $G$ of even order are actually involutions (see \cite[Theorem 38.2]{dorn}).
Furthermore, it is easy to see that  $g\in G$ is an involution if and only if $\tr(g)=0$. 
Hence, we obtain that for every involution $x\not \in C$ the coset $Cx$ contains exactly one involution.

Suppose that $q\equiv 3\pmod 4$ and use the following short argument that was communicated to us by R. M. Guralnick. The order of a Borel subgroup  $B$  of $G$ is $\frac{q(q-1)}{2}$ and so it is odd. 
Since the centralizer $C$ of an involution has order $q+1$ (see, for instance, \cite[II 8.4]{Hup}), the group $G$ can be written as the product $G=CB$ and the result follows.

So, we are left with the case $q\equiv 1 \pmod 4$. Let $j=\nu^{\frac{q-1}{4}}$ and define $t$ as the image in $\PSL(2,q)$ of the matrix
 $\bigl(\begin{smallmatrix} j & 0 \\0 & -j \end{smallmatrix}\bigr) \in  \widetilde G$.
Then, the centralizer $C$  of $t$ is the image in $G$ of the above subgroup $H$. Applying Proposition \ref{propH}, we obtain the result.
\end{proof}

\medskip
From Theorem \ref{cose} we deduce the following.

\begin{corollary}\label{invo}
Any two involutions in $\PSL(2,q)$ are conjugate by an element of odd order.
\end{corollary}

\begin{proof}
Let $t_1,t_2$ be two distinct involutions in $G=\PSL(2,q)$. As mentioned at the beginning of this section, there exist an element $x\in G$ such that $t_2=x^{-1}t_1x$. Let $C$ be the centralizer of $t_1$ in $G$. Theorem \ref{cose} implies that the coset $Cx$ contains an element $g$ of odd order. It is clear that 
$t_2=g^{-1}t_1g$.
\end{proof}

We close this section proving the following.

\begin{proposition}\label{PSLn}
Let $G=\SL(n,2)$ with $n\geq 4$. Then $G$ has an involution such that a coset of its centralizer consists only of elements of even order.
\end{proposition}

\begin{proof}
First, consider in $G_4=\SL(4,2)$ the following two involutions
$$t_1=\left(\begin{array}{cccc} 1 & 1  & 0 &0  \\0 & 1 & 0& 0\\ 0 & 0& 1 & 1\\0 & 0& 0& 1\end{array}
 \right),\qquad
 t_2=\left(\begin{array}{cccc} 1 & 0  & 1 &0  \\0 & 1 & 0& 0\\ 0 & 0& 1 & 0\\0 & 0& 0& 1\end{array}
 \right).$$

Let $C_4$ be the centralizer in $G_4$ of $t_1$. Then,
the coset $C_4t_2$ is the following set:
$$C_4t_2=\Bigg\{\left(\begin{array}{cccc}
a_1 & a_2  & a_1+a_3 &a_4  \\0 & a_1 & 0& a_3 \\ 
a_5 & a_6& a_5+a_7 & a_8\\ 0 & a_5& 0& a_7\end{array}   \right)  \mid a_1a_7 \neq  a_3a_5\Bigg\}.$$

It can be proved, using for instance \textsc{Magma} that these elements have order $2$, $4$ or $6$.

Now, assume $n>4$. In $G_n=\SL(n,2)$, we consider the  following block matrices:
$\tilde t_1=\diago(I_{n-4},t_1)$ and $\tilde t_2=\diago(I_{n-4},t_2)$, where $I_{n-4}$ denotes the identity matrix of size $n-4$. 
Clearly, these two elements are both involutions in $G_n$. Furthermore, denoting by $C_n$ the centralizer of $\tilde t_1$ in $G_n$, 
we see that the coset $C_n\tilde t_2$ consists of block matrices $g$ of shape $\bigl(\begin{smallmatrix} X_g & Y_g \\Z_g & W_g \end{smallmatrix}\bigr)$,
where $X_g\in \SL(n-4,2)$, $W_g\in C_4t_2$ and the matrices $Y_g=(y_{i,j})$ and $Z_g=(z_{i,j})$ are such that $y_{i,j}=0$ for $j=1,3$ and $z_{i,j}=0$ for $i=2,4$.

The particular shape of these matrices implies that if $g\in C_n\tilde t_2$ has order $k$, then the associated block $W_g\in C_4 t_2$
must satisfy the condition $(W_g)^k=\Big(\begin{smallmatrix} 1 & \ast & 0& \ast \\ 0 & 1 & 0& 0 \\
0& \ast &1 &\ast \\ 0 & 0 & 0& 1\end{smallmatrix}\Big)$. 
However this condition is never satisfied for $k$  odd. Since the orders of the elements in $C_4t_2$ are $2$, $4$ or $6$, it suffices to check only the cases $k=1,3,5$. This can be done using \textsc{Magma}. The claim now follows.
 \end{proof}

Observe that the involution $\tilde t_1$ described in the previous proposition does not belong to the center of a Sylow $2$-subgroup of the group.

\section{Proof of Theorem \ref{commu}}

In this section we prove Theorem \ref{commu} using the properties of the strongly real elements. An element $g$ of a group $H$ is called real if it is conjugate to its inverse and  is called strongly real if there exists an involution $t\in H$ such that $tgt=g^{-1}$. Observe that an element is strongly real if and only if it can be written as the product of two involutions.

\begin{lemma}
Any strongly real element of $\PSL(2,q)$ is a coprime commutator.
\end{lemma}

\begin{proof}
Let $g\in G=\PSL(2,q)$ be a strongly real element. Then, there exist two distinct involutions $t_1$ and $t_2$ in $G$ such that $g=t_1t_2$. 
By Corollary \ref{invo} there exist an element $x\in G$ of odd order such that $t_2=x^{-1} t_1 x$. Hence,
$$g=t_1t_2=t_1x^{-1} t_1 x=[t_1,x],$$
and so $g$ is a coprime commutator.
\end{proof}

Note that, actually, we proved that any strongly real element $g$ in $G=\PSL(2,q)$ can be written as $g=[a,b]$, for some involution $a$ and some element $b$ of odd order.
\medskip

Assume $q>3$.
If $q\not \equiv 3 \pmod 4$, then every element of $G$ is strongly real, see \cite{TZ} and \cite{Gill}. So, by the previous Lemma, it is a coprime commutator.

If $q=p^f \equiv 3 \pmod 4$, all real elements are actually strongly real \cite{Gill}. So, it suffices to study the non-real elements. By \cite[Theorem 38.1]{dorn}, only the two classes of unipotent elements are not real.
Furthermore, these elements have order $p$.

Let $P$ be a Sylow $p$-subgroup in $G$. Then $P$ is elementary abelian of order $q$ and $B=N_G(P)$, the Borel subgroup of $G$, is a Frobenius group with a cyclic complement of order $\frac{q-1}{2}$, which acts irreducibly on $P$. Hence,  every element of $P$ is of shape $[g,a]$, for some $g\in P$ and an element $a$ in the complement. The proof is now complete.

\end{document}